\documentclass[11pt]{article}
\usepackage{amsmath, amssymb, amsthm, cite}

\newtheorem{thm}{Theorem}[section]
\newtheorem{lem}[thm]{Lemma}
\newtheorem{problem}[thm]{Problem}
\numberwithin{equation}{section}

\makeatletter \@addtoreset{equation}{section} \makeatother

\newcommand{\tg}{\tilde{g}}
\newcommand{\tgamma}{\tilde{\gamma}}
\newcommand{\rmand}{\quad\hbox{ and }\quad}
\allowdisplaybreaks

\begin{document}

\begin{center}
{\Large\bf Surface embedding of $(n,k)$-extendable graphs}
\end{center}

\begin{center}
Hongliang Lu$^{1}$ and David G.L. Wang$^{2}$\\[6pt]

$^{1}$School of Mathematics and Statistics\\
Xi'an Jiaotong university, 710049 Xi'an, P. R. China\\
$^{2}$Department of Mathematics\\
University of Haifa, 3498838 Haifa, Israel

{\tt $^{1}$luhongliang@mail.xjtu.edu.cn}$\quad $ {\tt
$^{2}$wgl@math.haifa.ac.il}
\end{center}

\begin{abstract}
This paper is concerned with the surface embedding of matching extendable graphs. 
There are two directions extending the theory of perfect matchings, 
that is, matching extendability and factor-criticality. 
In solving a problem posed by Plummer,
Dean 
(￼The matching extendability of surfaces, J. Combin. Theory Ser. B 54 (1992), 133--141) 
established the fascinating formula
for the minimum number $k=\mu(\Sigma)$ such that every $\Sigma$-embeddable graph
is not $k$-extendable. Su and Zhang, Plummer and Zha 
found the minimum number $n=\rho(\Sigma)$ such that every
$\Sigma$-embeddable graph is not $n$-factor-critical. Based on the notion of $(n,k)$-graphs
which associates these two parameters, we found the formula for the minimum number
$k=\mu(n,\Sigma)$ such that every $\Sigma$-embeddable graph is not an $(n,k)$-graph.
To access this two-parameter-problem, 
we consider its dual problem and find out $\mu(n,\Sigma)$ conversely.
The same approach works for rediscovering the formula of the number~$\rho(\Sigma)$.
\end{abstract}




\section{Introduction}\label{sec:intro}

Since the notion of 
graph matching extendability was introduced by Plummer~\cite{Plu80} in~1980,
graph theory problems connecting the matching extendability 
have been studied quite extensively.
We recommend~\cite{Plu08BC,YL09B} for recent progress,
and the references therein for an extensive survey.
Among those beautiful structural theorems,
much attention has been paid to the interplay 
between the extendability and genus of the same graph.
For basic notions on topological graph theory, 
especially on graph embeddings,
we refer the reader to Gross and Tucker's book~\cite{GT87B}.
See also~\cite{BW09B} for a recent collection of topics in topological graph theory.

It was Nishizeki~\cite{Nis79} who initiated the study of
matching problems in relation with the genera of graphs.
He gave a lower bound of the cardinality of the maximum matchings
of a graph in terms of five basic parameters including the genus.
In the late 1980s, 
interests in matching extendability versus surface embeddings began 
with the charming result that no planar graph is $3$-extendable; see~\cite{Plu88PB}.
Considering both orientable and non-orientable surfaces,
Plummer~\cite{Plu88} studied the problem of determining 
the minimum integer $k$ 
such that every $\Sigma$-embeddable graph is not $k$-extendable. 
In~1992, Dean~\cite{Dean92} found the complete answer to this problem
based on some partial results obtained by Plummer.

\begin{thm}[Dean~\cite{Dean92}, Plummer~\cite{Plu88}]\label{thm:k}
Let $\Sigma$ be a surface of characteristic~$\chi$. 
Let $\mu(\Sigma)$ to be the minimum integer $k$ 
such that every $\Sigma$-embeddable graph is not $k$-extendable. Then
\begin{equation}\label{ans:mu}
\mu(\Sigma)=\begin{cases}
3,&\text{if $\Sigma$ is homeomorphic to the sphere};\\[3pt]
2+\lfloor \sqrt{4-2\chi(\Sigma)}\rfloor,&\text{otherwise}.
\end{cases}
\end{equation}
\end{thm}

Their proofs made a heavy use of the theory of Euler contributions,
which was due to Lebesgue~\cite{Leb40} in 1940,
and further developed by Ore~\cite{Ore67} in 1967, 
and by Ore and Plummer~\cite{OP69} in 1969.

Besides graph extendability which keeps the parity of the order of graphs to be even, 
factor-criticality is yet another direction that extends the theory of perfect matchings
which allows a graph to have an odd number of vertices.
Su and Zhang~\cite{SZ} considered the analogue problem 
relating the factor-criticality with surface embedding.
They found the answer for all surfaces except the Klein bottle,
which was settled by Plummer and Zha~\cite{PZ04}.

\begin{thm}[Su--Zhang~\cite{SZ}, Plummer--Zha~\cite{PZ04}]\label{thm:fc}
Let $\Sigma$ be a surface of characteristic~$\chi$.
Let~$\rho(\Sigma)$ be the minimum integer~$n$ 
such that every $\Sigma$-embeddable graph is not $n$-factor-critical.
Then we have 
\begin{equation}\label{ans:fc}
\rho(\Sigma)=\begin{cases}
5,&\text{if $\Sigma$ is homeomorphic to the sphere};\\
\lfloor(5+\sqrt{49-24\chi}\,)/2\rfloor,&\text{otherwise}.
\end{cases}
\end{equation}
\end{thm}

Citing the unpublished manuscript~\cite{SZ} for Theorem~\ref{thm:fc},
as it was done in~\cite{PZ04} and~\cite[Theorem 6.3.19]{YL09B},
we will give a proof to Theorem~\ref{thm:fc} in the appendix.

A more generalized notion, $(n,k)$-graphs, has been defined by Liu and Yu~\cite{LY01}, 
which associates the extendability parameter~$k$ in Theorem~\ref{thm:k}
and the factor-criticality parameter~$n$ in Theorem~\ref{thm:fc}.
A graph~$G$ is said to be an {\em $(n,k)$-graph} if for any vertex subset~$S$ of~$n$ vertices,
the subgraph $G-S$ is $k$-extendable.
Our main result is the following theorem.

\begin{thm}\label{thm:nk}
Let $n\ge1$ and let $\Sigma$ be a surface of characteristic~$\chi$.
Let $\mu(n,\Sigma)$ be the minimum integer~$k$ 
such that there is no $\Sigma$-embeddable $(n,k)$-graphs.
Then we have
\begin{equation}\label{ans:nk}
\mu(n,\Sigma)=\begin{cases}
\max(0,\,3-\lceil n/2\rceil),&\text{if $\Sigma$ is homeomorphic to the sphere};\\[3pt]
\max(0,\,\lfloor (\,7-2n+\sqrt{49-24\chi}\,)/4\rfloor),&\text{otherwise}.
\end{cases}•
\end{equation}
\end{thm}

To the best of our knowledge, 
this is the first time that an extendability-genus problem 
with more than one parameter is solved.
In fact, we often find the target $\mu(n,\Sigma)$ 
not easy to be understood at the first time when we deal with the problem.
On the way of making it more approachable,
we notice that
there have always been two dual approaches to graph embedding problems:
fix the graph and vary the surface, or fix the surface and vary the graph.
For example,
the Heawood Map Theorem~\cite{Hea1890}
can be viewed in either way: find the largest
complete graph embeddable in the surface of a given genus,
or find the lowest genus surface in which a given complete graph
can be embedded.
Robertson and Seymour's Kuratowski-type theorem~\cite{RS90},
which solves a problem of Erd\H{o}s and K\"onig,
is an example of fix-the-surface.
Inspired from the above idea, we formulate the following dual problem.
Denote by~$S_h$ the orientable surface of genus~$h$,
and denote by~$N_k$ the non-orientable surface of genus~$k$.

\begin{problem}\label{prb:nk2}
Let $n\ge1$ and $k\ge0$.
Define $g(n,k)$ to be the minimum integer 
such that there exists an $S_{g(n,k)}$-embeddable $(n,k)$-graph.
Define $\tg(n,k)$ to be the minimum integer 
such that there exists an $N_{\tg(n,k)}$-embeddable $(n,k)$-graph.
Find~$g(n,k)$ and~$\tg(n,k)$.
\end{problem}

By using the technique of Euler contribution and analyze the graph theoretical properties
of $(n,k)$-graphs, 
we find the complete answer to Problem~\ref{prb:nk2}; see Theorem~\ref{thm:ans:nk2}.
Then, Theorem~\ref{thm:nk} can be obtained 
by inversely computing the extendability $k$ in terms of the genus
and the factor-criticality according to the orientability of the surface.

Note that $(0,k)$-graphs are exactly $k$-extendable graphs,
while $(n,0)$-graphs are exactly $n$-factor-critical graphs.
We remark that, however, 
Formula~(\ref{ans:nk}) appears quite different from Formula~(\ref{ans:mu}).
The essential reason may be the fact that
a $k$-extendable graph may be bipartite 
while an $(n,k)$-graph with $n\ge1$ cannot be bipartite.
In contrast, Formula~(\ref{ans:fc}) looks like a relative of~(\ref{ans:nk}).
In fact, Formula~(\ref{ans:nk}) will be obtained from inversely solving some
inequalities which arised from the answer to Problem~\ref{prb:nk2}.
At the end of this paper, we include an appendix that
via the same approach one may retrieve Formula~(\ref{ans:fc}).

\section{Preliminaries}\label{sec:preliminary}

In this section, we recall some necessary concepts, notation and known results which
will be used in the subsequent sections.
Let~$G$ be a simple graph.
Denote by~$V(G)$ the set of vertices of~$G$,
and by~$E(G)$ the set of edges of~$G$.
The number $|V(G)|$ of vertices of~$G$ is denoted by~$|G|$ for short.
As usual,
we use the notation~$\delta(G)$ to denote the minimum degree of~$G$,
and use~$\kappa(G)$ to denote the connectivity of~$G$.
Let~$v$ be a vertex of~$G$.
A vertex~$u$ is said to be a neighbor of~$v$ if~$uv$ is an edge of~$G$.
We denote by~$N(v)$ the induced subgraph of~$G$ generated by
the neighbors of~$v$.

\subsection{The surface embedding}

Every surface is homeomorphic to an orientable surface or
a non-orientable surface.
Denote by~$S_h$ the orientable surface of genus~$h$,
and by~$N_k$ the non-orientable surface of genus~$k$.
Let~$\Sigma$ be a surface of genus~$g(\Sigma)$.
Then the Euler characteristic~$\chi(\Sigma)$ is defined to be
\[
\chi(\Sigma)=\begin{cases}
2-2g(\Sigma),&\text{if $\Sigma$ is orientable},\\
2-g(\Sigma),&\text{if $\Sigma$ is non-orientable}.
\end{cases}•
\]
Let $G$ be a graph.
We say that~$G$ is {\em $\Sigma$-embeddable} 
if it can be drawn on~$\Sigma$ without crossing itself.
The minimum value~$g$ such that~$G$ is $S_g$-embeddable
is said to be the {\em genus} of~$G$, denoted~$g(G)$.
Any embedding of~$G$ on~$S_{g(G)}$
is said to be a {\em minimal (orientable) embedding}.
Similarly,
the minimum value~$\tg$ such that~$G$ is $N_{\tg}$-embeddable
is said to be the {\em non-orientable genus} of~$G$, 
denoted~$\tg(G)$.
Any embedding of~$G$ on~$N_{\tg(G)}$
is said to be a {\em minimal (non-orientable) embedding}.
An embedding is said to be {\em $2$-cell} (or {\em cellular}) if 
every face is homeomorphic to an open disk.
Working on minimal embeddings, one should notice
the following two fundamental results 
which are due to Youngs~\cite{You63} and Parsons et al.~\cite{PPPV87} respectively.

\begin{thm}[Youngs~\cite{You63}]
Every minimal orientable embedding of a graph is $2$-cell.
\end{thm}

\begin{thm}[Parson--Pica--Pisanski--Ventre~\cite{PPPV87}]
Every graph has a minimal non-orientable embedding which is $2$-cell.
\end{thm}

Without special explanation,
we use the terminology ``embedding'' to mean a cellular embedding.

The formula of non-orientable genera of complete graphs was found by
Franklin~\cite{Fra34} in~1934 for the complete graph~$K_7$,
and by Ringel~\cite{Rin54} in~1954 for the other~$K_n$.
Early contributors include Heawood, Tietze, Kagno, Bose, Coxeter, Dirac, and so on;
see~\cite{Rin54}.
The more difficult problem of finding the genera of complete graphs
has been explored by Heffter, Ringel, Youngs, Gustin, Terry, Welch, Guy, Mayer, and so on.
A short history can be found in the well-known work~\cite{RY68} of Ringel and Youngs in~1968,
who settled the last case.
These formulas are as follows.

\begin{thm}\label{thm:g:K}
Let $n\ge5$. Then we have
\vskip3pt
\begin{itemize}
\item[(i)]
$\tg(K_7)=3$, and $\tg(K_n)=\lceil (n-3)(n-4)/6\rceil$ for $n\ne7$;
\vskip3pt
\item[(ii)] 
$g(K_n)=\lceil (n-3)(n-4)/12\rceil$.
\end{itemize}
\end{thm}

\subsection{The Euler contribution}
Let $G\to\Sigma$ be an embedding of a graph~$G$ on the surface~$\Sigma$.
Euler's formula states that $|G|-e+f=\chi(G)$, where $e$ is the number of edges of~$G$,
and $f$ is the number of faces of~$G$ in the embedding.
Let $x_i$ denote the size of the $i$th face that contains~$v$, i.e.,
the length of its boundary walk. 
The {\em Euler contribution} of~$v$ is defined to be
\[
\Phi(v)=1-{d(v)\over 2}+\sum_{i}{1\over x_i},
\]
where the sum ranges over all faces~$F_i$ containing~$v$.
One should keep in mind that 
a face may contribute more than one angle to a vertex. This can be seen from, for example,
the embedding of the complete graph~$K_5$ on the torus.
From Euler's formula, in any embedding of a connected graph~$G$,
we have $\sum_v\Phi(v)=\chi(\Sigma)$. So there exists a vertex such that 
\begin{equation}\label{def:Phi}
\Phi(v)\ge{\chi(\Sigma)\over |G|}.
\end{equation}
Such a vertex is said to be a {\em control point} of the embedding.
Definition~(\ref{def:Phi}) implies the following lemma immediately,
see~\cite[Lemma 2.5]{Plu88} and~\cite[Lemma 2.5]{Dean92}. 

\begin{lem}\label{lem:ctrl}
Let~$G$ be a connected graph of at least $3$ vertices.
Let $G\to\Sigma$ be an embedding.
Let~$v$ be a control point which is contained in~$x$ triangular faces.
Then we have
\begin{equation}\label{ineq:ctrl}
{d(v)\over 6}\le {d(v)\over 4}-{x\over 12}\le 1-{\chi(\Sigma)\over |G|}.
\end{equation}
\end{lem}

We will use it to find the lower bound of the numbers~$g(n,k)$ and~$\tg(n,k)$;
see the proofs of Theorem~\ref{thm:n>=6:k=0}, Theorem~\ref{thm:g},
and Theorem~\ref{thm:g=1}.

\subsection{The matching extendability}

Let $G$ be a graph and $k\ge0$.
A {\em $k$-matching} of~$G$ is a set of~$k$ pairwise disjoint edges.
For any $k$-matching~$M$,
we use the notation~$|M|$ to denote the number~$k$ of edges in~$M$.
It will not be confusing with the notation~$|G|$ from context,
which stands for the number of vertices.

A $k$-matching is said to be {\em perfect} if~$G$ has exactly~$2k$ vertices.
The graph~$G$ is said to be {\em $k$-extendable} if
it has a perfect matching, and for any $k$-matching~$M$, 
the graph~$G$ has a perfect matching containing~$M$.
Plummer~\cite{Plu80} gave the following basic facts on 
matching extendability.

\begin{thm}[Plummer~\cite{Plu80}]\label{thm:ext:basic}
Let $k\ge0$ and let~$G$ be a connected $k$-extendable graph. Then we have
\begin{itemize}
\item[(i)]
if $k\ge1$, then the graph~$G$ is $(k-1)$-extendable;
\vskip 2pt
\item[(ii)]
the graph~$G$ is $(k+1)$-connected and thus $\delta(G)\ge k+1$.
\end{itemize}
\end{thm}

Next is a slightly deeper result, which is due to Lou and Yu~\cite[Theorem 7]{LY04};
see also~\cite[Chap.~6]{YL09B}. 

\begin{thm}[Lou--Yu~\cite{LY04}]\label{thm:LY:bip}
If~$G$ is a $k$-extendable graph of order at most~$4k$,
then either the graph~$G$ is bipartite or the connectivity $\kappa(G)$ of~$G$ is at least~$2k$.
\end{thm}

We will use it in the proof of Theorem~\ref{thm:g}.

\subsection{The $(n,k)$-graphs}

Let~$G$ be a graph. Let $n$ and~$k$ be nonnegative integers such that
\begin{equation}\label{def:nk}
|G|\ge n+2k+2
\rmand
|G|-n\equiv 0\pmod2.
\end{equation}
Liu and Yu~\cite{LY01} introduced the concept of an {\em $(n,k)$-graph}
which was defined to be a graph~$G$ such that
if every subgraph of~$G$ obtained by deleting $n$ vertices is $k$-extendable.
In particular, $(n,0)$-graphs are $n$-factor-critical graphs,
and $(0,k)$-graphs are $k$-extendable graphs.
The following basic results about $(n,k)$-graphs 
can be found in~\cite[Proposition~3.5 and Theorem~3.8]{LY01}.

\begin{thm}[Yu--Liu~\cite{LY01}]\label{thm:nk:basic}
Let~$G$ be a connected $(n,k)$-graph. Then we have
\begin{itemize}
\item[(i)]
If $n\ge2$, then~$G$ is an $(n-2,\,k+1)$-graph.
\item[(ii)]
If $k\ge1$, then $\delta(G)\ge\kappa(G)\ge n+k+1$.
\end{itemize}
\end{thm}

We will use Theorem~\ref{thm:nk:basic} to deal with the sporadic cases in~\S\ref{sec:nk}.

\section{Determining the value $\mu(n,\Sigma)$}\label{sec:nk}

Recall that $\mu(n,\Sigma)$ is the minimum integer~$k$ 
such that there is no $\Sigma$-embeddable $(n,k)$-graphs.
In this section, we will solve Problem~\ref{prb:nk2} and 
express it inversely to obtain the formula for $\mu(n,\Sigma)$.

First of all, since $(n,0)$-graphs are exactly $n$-factor-critical graphs, 
the numbers $g(n,0)$ and $\tg(n,0)$ can be 
deduced from Theorem~\ref{thm:fc} inversely.
To warm up and for self-completeness, we give a brief proof for this case when $n\ge6$.

\begin{thm}\label{thm:n>=6:k=0}
Let $n\ge6$. Then we have 
\[
g(n,0)=\biggl\lceil {(n-1)(n-2)\over 12}\biggr\rceil
\rmand
\tg(n,0)=\biggl\lceil {(n-1)(n-2)\over 6}\biggr\rceil.
\]
\end{thm}

\begin{proof}
We will show the formula for the number~$\tg(n,0)$. 
The formula for the number~$g(n,0)$ can be proved similarly.
Denote 
\[
\tgamma=\biggl\lceil {(n-1)(n-2)\over 6}\biggr\rceil.
\]
By Theorem~\ref{thm:g:K}, the complete graph~$K_{n+2}$ has non-orientable genus~$\tgamma$.
Since~$K_{n+2}$ is $n$-factor-critical, we have $\tg(n,0)\leq\tgamma$ from the definition.
Let $\Sigma=N_{\tg(n,0)}$.
Assume that there exists a $\Sigma$-embeddable 
$n$-factor-critical graph~$G$
such that $\tg(\Sigma)\le\tgamma-1$. It follows that
\begin{equation}\label{pf:k=0:1}
\chi(\Sigma)=2-\tg(\Sigma)\ge 3-\tgamma.
\end{equation}
Let $v$ be a control point in an embedding $G\to\Sigma$ and denote $y=d(v)$. 
Note that every $n$-factor-critical graph has minimum degree at least $n+1$.
So $y\ge n+1$ and $|G|\ge n+2$. By Lemma~\ref{lem:ctrl}, we have
\begin{equation}\label{pf:k=0:2}
\frac{n+1}{6}
\leq\frac{y}{6}
\leq 1-\frac{\chi(\Sigma)}{|G|}.
\end{equation}
Since $n\ge6$, we infer that $\chi(\Sigma)<0$ from~(\ref{pf:k=0:2}).
On the other hand, if $|G|=n+2$, then $G=K_{n+2}$ since it is $n$-factor-critical.
It follows that $g(G)=\tgamma$, a contradiction. 
Therefore, we have
\begin{equation}\label{pf:k=0:3}
|G|\ge n+4
\end{equation}
by parity argument.
Since $\chi(\Sigma)<0$, 
substituting~(\ref{pf:k=0:1}) and~(\ref{pf:k=0:3}) into~(\ref{pf:k=0:2}), we obtain that
\[
\frac{n+1}{6}
\leq 1-\frac{\chi(\Sigma)}{|G|}
\leq 1-\frac{3-\tgamma}{n+4}
\leq 1-\frac{3-[(n-1)(n-2)+4]/6}{n+4},
\]
which implies $n\leq 4$, a contradiction.

Along the same line, 
one may easily obtain the orientable genus.
This completes the proof.
\end{proof}

To handle Problem~\ref{prb:nk2} for most of the other pairs~$(n,k)$, we will need the next two lemmas. 

\begin{lem}\label{lem:nk:bip}
Let $k\ge1$. 
Let~$G$ be a connected $(n,k)$-graph.
Then~$G$ does not have an induced bipartite subgraph with more than $|G|-n-1$ vertices.
\end{lem}

\begin{proof}
Suppose to the contrary that the graph~$G$ has an
induced bipartite subgraph~$H$ with bipartition $(U,W)$ such that $|H|\ge |G|-n$.
We can suppose without loss of generality that $|H|=|G|-n$.
Let $S=G-H$. Then $|S|=n$.
Since~$G$ is an $(n,k)$-graph,
the subgraph~$H$ has a perfect matching.
So $|U|=|W|$.
Since $k\ge1$, we deduce from~(\ref{def:nk}) that $|G|\ge n+4$.
It follows that $|U|=|W|\ge 2$.
Since the graph~$G$ is connected,
we can suppose that there is an edge connecting
a vertex~$x$ in the subgraph~$S$ and a vertex~$y$ in the subgraph~$U$, without loss of generality.
Let~$z$ be a vertex in~$U$ other than the vertex~$y$.
Let 
\[
S'=G\bigl[V(S)\cup\{z\}\backslash\{x\}\bigr]
\]
be the induced graph obtained by
replacing the vertex~$x$ in~$S$ by~$z$.
It is obvious that $|S'|=n$ and no perfect matching of $G-S'$
contains the edge~$xy$, contradicting
the hypothesis that~$G$ is an $(n,k)$-graph with $k\ge1$.
This completes the proof.
\end{proof}

Lemma~\ref{lem:d:x} is a lower bound of the degree of any vertex~$v$ in terms 
of the number of triangular faces containing~$v$.

\begin{lem}\label{lem:d:x}
Let $n\ge0$ and $k\ge1$.
Let~$G$ be a connected $(n,k)$-graph embedded on a surface~$\Sigma$.
Let~$v$ be a vertex of~$G$ which is contained in~$x$ triangular faces.
Then we have
\[
d(v)\ge\begin{cases}
n+k+1+\lceil x/2\rceil,&\text{if $x\le 2k-2$},\\
n+2k+1,&\text{if $x\ge 2k-1$}.
\end{cases}
\]
\end{lem}

\begin{proof}
We will always use the following bound of the number~$|G|$ from~(\ref{def:nk}):
\begin{equation}\label{pf:lb:|G|}
|G|\ge n+2k+2.
\end{equation}
Assume that $x\geq 2k-1$.
We claim that the set~$N(v)$ has a $k$-matching.
If the vertex~$v$ is contained in a non-triangular face,
then the set~$N(v)$ has a $\lceil x/2\rceil$-matching, and thus a $k$-matching.
Otherwise all faces containing the vertex~$v$ are triangles.
Thus, we have $|N(v)|=x$.
If $x\ge 2k$, then the claim is true.
Otherwise $d(v)=x=2k-1$.
It follows that the set~$N(v)$ contains a Hamilton cycle.
Since the graph~$G$ is connected,
there exists an edge~$uw$ such that $u\in N(v)$
and~$w$ is a vertex in the subgraph $G-N(v)-v$. 
Since $|N(v)-u|=2k-2$, 
the subgraph $N(v)-u$ has a $(k-1)$-matching, say,~$M'$.
Note that 
\[
|G-N(v)-v-w|=|G|-2k-1\ge n+1.
\]
Let~$S$ be a set of~$n$ vertices of~$G-N(v)-v-w$.
Then the graph induced by $G-S-V(M'\cup \{uw\})$ has no perfect matchings,
since the vertex~$v$ in it is isolated. This contradicts the condition that~$G$
is an $(n,k)$-graph, and confirms the claim.

Let~$M$ be a $k$-matching of~$N(v)$. Let 
\[
H=N(v)-V(M).
\] 
Suppose to the contrary that $d(v)\le n+2k$.
It follows immediately that $0\le|H|\le n$.
By~(\ref{pf:lb:|G|}), we have
\[
|G-v-V(M)|=|G|-1-2k\ge n+1.
\]
Therefore, there exists a vertex subset~$S$ of $G-v-N(v)$ such that $|S|=n-|H|$.
Let $S'=S\cup V(H)$.
Then $|S'|=n$. Since~$G$ is an $(n,k)$-graph,
the subgraph $G-S'$ is $k$-extendable. However,
the $k$-matching~$M$ in $G-S'$ is not extendable because the vertex~$v$ is isolated,
a contradiction.

Below we can suppose that $x\le 2k-2$.
Define 
\[
h=k-\lceil x/2\rceil. 
\]
Then $h\ge1$.
Suppose to the contrary that
\begin{equation}\label{ineq:d}
d(v)\le n+k+\lceil x/2\rceil.
\end{equation}
Note that the subgraph~$N(v)$ has an $\lceil x/2\rceil$-matching.
Let $M$ be such a matching.
Since deleting any two adjacent vertices from an $(n,k)$-graph results in an $(n,k-1)$-graph, 
the graph $G-V(M)$ is an $(n,k-\lceil x/2\rceil)$-graph.
By Theorem~\ref{thm:nk:basic}, we have
\[
\delta(G-V(M))\ge n+k-\lceil x/2\rceil+1.
\] 
Hence, the degree
\[
d(v)\ge \delta(G-V(M))+|V(M)|=n+k+\lceil x/2\rceil+1.
\]
This completes the proof.
\end{proof}

Here comes the answer of Problem~\ref{prb:nk2} 
when neither the number~$n$ nor~$k$ is too small.

\begin{thm}\label{thm:g}
Let $n,k\geq 1$ and $n+2k\ge6$. Then we have
\[
\tg(n,k)=\biggl\lceil{(n+2k-1)(n+2k-2)\over6}\biggr\rceil
\rmand
g(n,k)=\biggl\lceil{(n+2k-1)(n+2k-2)\over12}\biggr\rceil.
\]
\end{thm}

\begin{proof}
The proofs for the two formulas are similar.
Since in the proof of Theorem~\ref{thm:n>=6:k=0} 
we gave details for the formula of the number~$\tg(n,0)$,
in this proof we will show the one for the number~$g(n,k)$ in detail.
Denote 
\[
\gamma=\biggl\lceil{(n+2k-1)(n+2k-2)\over12}\biggr\rceil.
\]
By Theorem~\ref{thm:g:K}, the complete graph~$K_{n+2k+2}$ has genus~$\gamma$.
It follows that $g(n,k)\leq\gamma$.
Assume that there exists an $(n,k)$-graph~$G$ 
embeddable on an orientable surface~$\Sigma$ of genus at most
$\gamma-1$.
Since $n+2k\ge6$, we have $\gamma\ge2$ and thus
\begin{equation}\label{pf:n+2k>=6:1}
\chi(\Sigma)=2-2g(\Sigma)\ge4-2\gamma.
\end{equation}
Let $v$ be a control point in an embedding $G\to\Sigma$. 
Let~$x$ be the number of triangular faces containing~$v$. Write $y=d(v)$.

Assume that $|G|\ge n+4k+1$.
If $x\geq 2k-1$, then Lemma~\ref{lem:d:x} implies that $y\ge n+2k+1$.
By Lemma~\ref{lem:ctrl}, we deduce that
\begin{equation}\label{pf:n+2n>=6:2}
{n+2k+1\over 6}
\le {y\over 6}
\le 1-{\chi(\Sigma)\over |G|}.
\end{equation}
Since $n+2k\ge6$, we infer from~(\ref{pf:n+2n>=6:2}) that $\chi(\Sigma)<0$.
Consequently, substituting the inequality~(\ref{pf:n+2k>=6:1}) 
into the above inequality gives
\[
{n+2k+1\over 6}
\le 1-{4-2\gamma\over n+4k+1}
\le 1-{4-2{[(n+2k-1)(n+2k-2)+10]/12}\over n+4k+1}.
\]
Simplifying it we get
\[
(2k-3)^2+(2k-1)n\le 2.
\]
It follows that $n=k=1$, 
contradicting the condition $n+2k\ge6$.
Otherwise $x\leq 2k-2$, then Lemma~\ref{lem:d:x} gives 
\[
y\ge n+k+1+\lceil x/2\rceil.
\]
By the second inequality in~(\ref{ineq:ctrl}), we deduce that
\begin{equation}\label{pf:n+2k>=6:4}
\frac{n+k+1}{4}
\le{n+k+1+\lceil x/2\rceil\over 4}-\frac{x}{12}
\le\frac{y}{4}-\frac{x}{12}
\le1-{\chi(\Sigma)\over |G|}.
\end{equation}
If $\chi(\Sigma)\ge0$, then~(\ref{pf:n+2k>=6:4}) implies that $n+k\le3$, 
contradicting the given condition 
that $n,k\ge1$ and $n+2k\ge6$. Thus we have $\chi(\Sigma)<0$.
Consequently, substituting the inequality~(\ref{pf:n+2k>=6:1}) into~(\ref{pf:n+2k>=6:4}), we get
\[
\frac{n+k+1}{4}
\le 1-{4-2\gamma\over n+4k+1}
\le 1-{4-2{[(n+2k-1)(n+2k-2)+10]/12}\over n+4k+1}.
\]
Simplifying it we get
\[
4k^2+7(n-3)k+n^2+15\le 0,
\]
contradicting the conditions $n,k\ge1$ again.

Therefore, we deduce that $|G|\leq n+4k$.
Let $S$ be a vertex subset of~$N(v)$ such that $|S|=n$.
Then the subgraph $G-S$ is $k$-extendable from the definition.
On the other hand, Lemma~\ref{lem:nk:bip} implies that the subgraph $G-S$ is not bipartite.
By Theorem~\ref{thm:LY:bip}, we have $\kappa(G-S)\geq 2k$.
Thus $y\geq n+2k$. We claim that
\begin{equation}\label{clm:1}
\frac{y}{4}-\frac{x}{12}\ge{n+2k+1\over 6}.
\end{equation}
In fact, if $y=n+2k$, then Lemma~\ref{lem:d:x} implies that $x\leq 2k-2$.
We can derive that
\[
\frac{y}{4}-\frac{x}{12}
\ge {n+2k\over 4}-{2k-2\over 12}
\ge {n+2k+1\over 6}.
\]
Otherwise $y\ge n+2k+1$, then the trivial condition $x\le y$ implies
\[
\frac{y}{4}-\frac{x}{12}
\ge \frac{y}{4}-\frac{y}{12}
\ge {n+2k+1\over 6}.
\]
This proves Claim~(\ref{clm:1}). 
Again, Lemma~\ref{lem:ctrl} implies $\chi(\Sigma)<0$. Since $|G|\ge n+2k+2$, we have
\begin{align*}
{n+2k+1\over 6}
\le \frac{y}{4}-\frac{x}{12}
\le1-{\chi(\Sigma)\over |G|}
\leq 1-{4-2{[(n+2k-1)(n+2k-2)+10]/12}\over n+2k+2}.
\end{align*}
Simplifying it gives $-10\le -12$, which is absurd.

The formula for the number~$\tg(n,k)$ can be proved along the same line.
This completes the proof.
\end{proof}

In view of the technical conditions ``$n\ge6$ and $k=0$'' in Theorem~\ref{thm:n>=6:k=0}, 
and ``$n,k\ge1$ and $n+2k\ge6$'' in Theorem~\ref{thm:g},
we are led to the sporadic cases that $(n,k)$ is one of the following $9$ pairs:
\begin{equation}\label{9pairs}
(5,0),\ (4,0),\ (3,0),\ (2,0),\ (1,0),\ (3,1),\ (1,2),\ (1,1),\ (2,1).
\end{equation}

The next result~\cite[Theorem 2.1]{LY04-1} will
make our proof to Theorem~\ref{thm:g=0:tg=1} simpler.

\begin{lem}[Lou--Yu~\cite{LY04-1}]\label{thm:5conn:4fc}
Any $5$-connected planar graph of even order is $4$-factor-critical.
\end{lem}

In fact, $5$-connected planar graphs of even order exist. 
This can be seen from a nice result of Zaks~\cite{Zaks76},
who constructed a $5$-regular $5$-connected planar graph~$G$ of order $66$
which has an edge belonging to the union of every pair of edge-disjoint Hamiltonian circuits
of~$G$. 

Six of the pairs in~(\ref{9pairs}) are solved in the following theorem.

\begin{thm}\label{thm:g=0:tg=1}
We have $g(n,k)=0$ and $\tg(n,k)=1$ for each of the following pairs~$(n,k)$:
\begin{equation}\label{pairs:nk1}
(4,0),\ (3,0),\ (2,0),\ (1,0),\ (1,1),\ (2,1).
\end{equation}
\end{thm}

\begin{proof}
Since the minimum genus of any orientable surface is~$0$,
and the minimum non-orientable genus of any non-orientable surface is~$1$,
it suffices to construct an $S_0$-embeddable $(n,k)$-graph,
and an $N_1$-embeddable $(n,k)$-graph for each pair~$(n,k)$ in~(\ref{pairs:nk1}).

Let~$G$ be a $5$-connected planar graph of even order.
The existence of such a graph can be seen from~\cite{Zaks76}.
By Lemma~\ref{thm:5conn:4fc}, we deduce that~$G$ is $4$-factor-critical,
namely a $(4,0)$-graph.
By Theorem~\ref{thm:nk:basic}~(i),
we derive that~$G$ is a $(2,1)$-graph,
and thus $2$-factor-critical. Thus
\begin{equation}\label{pf:g=0:1}
g(4,0)=g(2,1)=g(2,0)=0.
\end{equation}
We note that, when $n\ge1$, any graph obtained from an $(n,k)$-graph
by deleting a vertex is an $(n-1,k)$-graph. 
It is obvious that the operation of deleting an edge preserves the planarity of graphs.
Therefore, Result~(\ref{pf:g=0:1}) implies
\[
g(3,0)=g(1,1)=g(1,0)=0.
\]

For the non-orientable genera, it is clear that the complete graph~$K_6$ is $4$-factor-critical. 
On the other hand, 
$K_6$ has non-orientable genus~$1$ by Theorem~\ref{thm:g:K}.
For the same reasons as before, we deduce that 
$\tg(4,0)=\tg(2,1)=\tg(2,0)=1$,
and thus $\tg(3,0)=\tg(1,1)=\tg(1,0)=1$ since 
deleting an edge preserves the $N_1$-embeddability of graphs.
This completes the proof.
\end{proof}

Dealing with the remaining $3$ pairs $(n,k)$ for the non-orientable genera, 
we need the existence of an $N_2$-embeddable $5$-factor-critical graph,
which was proved constructively by Plummer and Zha~\cite{PZ04}. 
The next theorem solve the last case for Problem~\ref{prb:nk2}.

\begin{thm}\label{thm:g=1}
We have $g(n,k)=1$ and $\tg(n,k)=2$ for each of the following pairs $(n,k)$:
\begin{equation}\label{pairs:nk2}
(5,0),\ (3,1),\ (1,2).
\end{equation}
\end{thm}

\begin{proof}
Let $(n,k)$ be a pair in~(\ref{pairs:nk2}).

First, we show that $g(n,k)=1$.
By Theorem~\ref{thm:g:K}, the complete graph~$K_7$ has genus~$1$.
It is obvious that $K_7$ is a $(5,0)$-graph.
By Theorem~\ref{thm:nk:basic}~(i), the graph~$K_7$ is both a $(3,1)$-graph and a $(1,2)$-graph.
So it suffices to show that there is no $S_0$-embeddable $(1,2)$-graphs.

Suppose to the contrary that~$G$ is an $S_0$-embeddable $(1,2)$-graph, 
which is connected without loss of generality.
By Theorem~\ref{thm:nk:basic}~(ii), we have $\delta(G)\ge4$.
Let~$v$ be a control point in an embedding $G\to S_0$. Denote $y=d(v)$.
Let~$x$ be the number of triangular faces containing~$v$.
By Lemma~\ref{lem:ctrl}, we have $y\le 5$ since $\chi(S_0)=2>0$. Therefore $y\in\{4,5\}$.

Suppose that $y=5$.
If $x\ge 3$, then $N(v)$ has a $2$-matching, say,~$M$.
Let $u$ be the vertex constituting $N(v)-V(M)$.
Then the subgraph $G-u$ is $2$-extendable since~$G$ is a $(1,2)$-graph.
On the other hand, the $2$-matching~$M$ is not extendable in $G-u$
since the vertex~$v$ is isolated.
This contradiction implies that $x\le 2$.
Using the second inequality in~(\ref{ineq:ctrl}), we deduce that
\[
{13\over 12}={5\over 4}-{2\over 12}\le {y\over 4}-{x\over 12}\le 1-{\chi(S_0)\over |G|}
=1-{2\over |G|}<1,
\]
a contradiction.

Therefore, we have $y=4$.
Suppose that 
$N(v)=\{v_1,v_2,v_3,v_4\}$. 
Assume that $v_1v_2$ is an edge.
Since~$G$ is a $(1,2)$-graph, we have $\delta(G)\ge4$ by Theorem~\ref{thm:nk:basic} (ii). 
So there exists a vertex~$u$ in the subgraph $G-N(v)-v$ such that
$v_3u$ is an edge of~$G$. 
Now, the graph obtained by deleting the vertex~$v_4$ and the edges~$v_1v_2$ and~$v_3u$
does not have a perfect matching because the vertex~$v$ is isolated, a contradiction!
Since the edge~$v_1v_2$ is chosen arbitrarily, we conclude that the set~$N(v)$ is independent,
that is, $x=0$.
Now, the second inequality in~(\ref{ineq:ctrl}) gives
\[
1={y\over 4}={y\over 4}-{x\over 12}\le 1-{\chi(S_0)\over |G|}<1,
\]
a contradiction.

Now we come to show that $\tg(n,k)=2$.
Let~$G$ be an $N_2$-embeddable $(5,0)$-graph.
The existence of such a graph can be seen from~\cite{PZ04}.
By Theorem~\ref{thm:nk:basic}~(i), 
we see that $\tg(n,k)\le2$, and it suffices to show that there is no 
$N_1$-embeddable $(1,2)$-graphs. 
The remaining proof is similar to the one for orientable genera.
Suppose to the contrary that~$G$ is a connected $N_1$-embeddable $(1,2)$-graph.
Let~$v$ be a control point in an embedding $G\to N_1$.
Let~$x$ be the number of triangular faces containing~$v$.
Since $\chi(N_1)=1>0$, we have $d(v)\in\{4,5\}$.
If $d(v)=5$, then $x\le2$ and thus
\[
{13\over 12}\le 1-{\chi(N_1)\over |G|}
=1-{1\over |G|}<1,
\]
a contradiction. So $d(v)=4$, $x=0$, and thus
\[
1\le1-{\chi(N_1)\over |G|}<1,
\]
a contradiction.
This completes the proof.
\end{proof}

So far we have completely solved Problem~\ref{prb:nk2}.
Here is a summary of the solution.

\begin{thm}\label{thm:ans:nk2}
Let $n\ge1$ and $k\ge0$. Then we have
\begin{align}
g(n,k)&=\begin{cases}
0,&\text{if $n+2k\le 4$},\\
\lceil (n+2k-1)(n+2k-2)/12 \rceil,&\text{otherwise};
\end{cases}\label{expr:g}\\[5pt]
\tg(n,k)&=\begin{cases}
1,&\text{ if $n+2k\le 4$},\\
\lceil (n+2k-1)(n+2k-2)/6\rceil,&\text{ otherwise}.
\end{cases}\label{expr:tg}
\end{align}
\end{thm}

By Formula~(\ref{expr:g}), we can draw Table~\ref{tab:g} for $g(n,k)$ when~$n$ and~$k$ are small.

\begin{table}[htdp]
\caption{The value $g(n,k)$ (the minimum integer $g$
such that there exists an $S_g$-embeddable $(n,k)$-graph) 
for $1\le n\le 8$ and $0\le k\le 8$.}
\begin{center}
\begin{tabular}{|c|c|c|c|c|c|c|c|c|c|}
\hline
&0&1&2&3&4&5&6&7&8\\
\hline
1&0&0&1&3&5&8&11&16&20\\
\hline
2&0&0&2&4&6&10&13&18&23\\
\hline
3&0&1&3&5&8&11&16&20&26\\
\hline
4&0&2&4&6&10&13&18&23&29\\
\hline
5&1&3&5&8&11&16&20&26&32\\
\hline
6&2&4&6&10&13&18&23&29&35\\
\hline
7&3&5&8&11&16&20&26&32&39\\
\hline
8&4&6&10&13&18&23&29&35&43\\
\hline
\end{tabular}
\end{center}
\label{tab:g}
\end{table}

from the definition and Table~\ref{tab:g}, one may obtain Table~\ref{tab:mu} 
for the value $\mu(n,\Sigma)$ when the surface~$\Sigma$ is orientable. 
For example, we have $\mu(6,S_8)=3$ since
\[
g(6,2)=6<8<10=g(6,3).
\]

\begin{table}[htdp]
\caption
{The value $\mu(n,S_g)$ (the minimum~$k$ such that there is no $S_g$-embeddable $(n,k)$-graphs) 
for $1\le n\le 8$ and $0\le g\le 16$.}
\begin{center}
\begin{tabular}{|c|c|c|c|c|c|c|c|c|c|c|c|c|c|c|c|c|c|}
\hline
&0&1&2&3&4&5&6&7&8&9&10&11&12&13&14&15&16\\
\hline
1&2&3&3&4&4&5&5&5&6&6&6&7&7&7&7&7&8\\
\hline
2&2&2&3&3&4&4&5&5&5&5&6&6&6&7&7&7&7\\
\hline
3&1&2&2&3&3&4&4&4&5&5&5&6&6&6&6&6&7\\
\hline
4&1&1&2&2&3&3&4&4&4&4&5&5&5&6&6&6&6\\
\hline
5&0&1&1&2&2&3&3&3&4&4&4&5&5&5&5&5&6\\
\hline
6&0&0&1&1&2&2&3&3&3&3&4&4&4&5&5&5&5\\
\hline
7&0&0&0&1&1&2&2&2&3&3&3&4&4&4&4&4&5\\
\hline
8&0&0&0&0&1&1&2&2&2&2&3&3&3&4&4&4&4\\
\hline
\end{tabular}
\end{center}
\label{tab:mu}
\end{table}

In general, we can deduce a formula for the number~$\mu(n,\Sigma)$ by Theorem~\ref{thm:ans:nk2}.
As will be seen, the following lemma turns out to be considerably useful 
in solving inequalities involving the ceiling function.

\begin{lem}\label{lem:ineq:ceil}
Let $g$ be an integer. Let $x$ and~$y$ be real numbers. Then we have 
\begin{itemize}
\item[(i)]
$\lceil x\rceil\le g$ if and only if $x\le g$;
\vskip 2pt
\item[(ii)]
$g\le \lceil y\rceil -1$ if and only if $g<y$.
\end{itemize}
\end{lem}

It is elementary to show
Lemma~\ref{lem:ineq:ceil} and we omit its proof.
Here is the answer to the dual problem of Problem~\ref{prb:nk2}.

\begin{thm}
Let $n,\tg\ge1$ and $g\ge0$. Then we have
\begin{align}
\mu(n,S_g)&=\begin{cases}
\max\bigl(0,\,3-\lceil n/2\rceil\bigr),&\text{if $g=0$},\\[3pt]
\max\bigl(0,\,\lfloor (\,7-2n+\sqrt{1+48g}\,)/4\rfloor\bigr),&\text{if $g\ge1$};
\end{cases}\label{fm:mu:g}\\[4pt]
\tilde\mu(n,N_{\tg})
&=\max\Biggl(\,0,\ \biggl\lfloor{7-2n+\sqrt{1+24\tg}\over4}\biggr\rfloor\,\Biggr).\label{fm:mu:tg}
\end{align}
\end{thm}

\begin{proof}
The proofs for Formula~(\ref{fm:mu:g}) and Formula~(\ref{fm:mu:tg}) are highly similar. 
We will show Formula~(\ref{fm:mu:g}) only.

It is easy to observe two rules from Table~\ref{tab:g} and Table~\ref{tab:mu}.
First, in Table~\ref{tab:g}, 
the $n$th row is obtained by 
translating the $(n-2)$th row by one unit to the left for all $n\ge3$.
This is true in general, because of the bi-linear expression in the ceiling function in~(\ref{expr:g}).
Second, the number of entries~$i$ in the $n$th row of Table~\ref{tab:mu}
equals the difference between the $(n,i)$-entry and $(n,i-1)$-entry in Table~\ref{tab:g}.
This can be proved in general easily from the definitions of the numbers~$g(n,k)$ and~$\mu(n,S_g)$.

Consequently, we can deduce that in Table~\ref{tab:mu}, 
the $n$th row is obtained by subtracting~$1$ 
from the $(n-2)$th row, and by forcing it to be~$0$ if it is negative.
In other words, we have
\begin{equation}\label{rec:mu}
\mu(n,S_g)=\max\bigl(0,\,\mu(n-2,S_g)-1\bigr)\qquad\text{for all $n\ge3$.}
\end{equation}
So it suffices to find out the formulas for the numbers~$\mu(1,S_g)$ and~$\mu(2,S_g)$. 
Now we seek them respectively.

For any $i\ge2$, 
let $a_i$ be the right-most column label~$g$ 
such that the $(1,g)$-entry in Table~\ref{tab:mu} is~$i$. In other words,
\[
a_i=\max\{g\,\colon\,\mu(1,g)=i\}.
\]
For example, $a_2=0$ can be read from Table~\ref{tab:mu} directly.
From the above description of the way of constructing Table~\ref{tab:mu} 
from Table~\ref{tab:g}, we know that the number~$a_i$ exists for all $i\ge2$. 
Moreover, the number of entries of value~$i$ equals the difference $a_i-a_{i-1}$.
On the other hand, we have observed that this number should be the difference 
$g(1,i)-g(1,i-1)$. Therefore, we have 
\[
a_i-a_{i-1}=g(1,i)-g(1,i-1)\qquad \text{for all $i\ge3$}.
\]
By~(\ref{expr:g}), adding up these equalities gives
\[
a_k
=a_2+\sum_{i=3}^k(a_i-a_{i-1})
=\sum_{i=3}^k\bigl(g(1,i)-g(1,i-1)\bigr)
=g(1,k)-g(1,2)
=\biggl\lceil {k(2k-1)\over6}\biggr\rceil-1,
\]
for all $k\ge2$.
from the definition, we have $\mu(1,S_g)=k$ if and only if $a_{k-1}+1\le g\le a_k$, that is,
\begin{equation}\label{pf:mu:g:1}
\biggl\lceil {(k-1)(2k-3)\over6}\biggr\rceil
\le g
\le \biggl\lceil {k(2k-1)\over6}\biggr\rceil-1.
\end{equation}
Note that $k\ge2$ implies $g\ge1$ from~(\ref{pf:mu:g:1}). 
By Lemma~\ref{lem:ineq:ceil}, (\ref{pf:mu:g:1}) is equivalent to 
\[
{(k-1)(2k-3)\over6}
\le g
<{k(2k-1)\over6}.
\]
Solving $k$ in terms of $g$, we get
\[
{1+\sqrt{1+48g}\over 4}<k\le {5+\sqrt{1+48g}\over 4}.
\]
Since the above two bounds of $k$ form an interval of length~$1$, 
and since $k$ is an integer, we infer that 
\begin{equation}\label{sol:mu:n=1}
\mu(1,S_g)=k=\biggl\lfloor {5+\sqrt{1+48g}\over 4}\biggr\rfloor.
\end{equation}

Along the same line, we compute
\begin{equation}\label{sol:mu:n=2}
\mu(2,S_g)=k=\biggl\lfloor {3+\sqrt{1+48g}\over 4}\biggr\rfloor.
\end{equation}
Combining Eqs.~(\ref{rec:mu}), (\ref{sol:mu:n=1}), and~(\ref{sol:mu:n=2}),
we find Formula~(\ref{fm:mu:g}) for the case $g\ge1$. 
The expression for $g=0$ can be checked straightforwardly from Table~\ref{tab:mu}.
This proves Formula~(\ref{fm:mu:g}).

Similarly, one may draw Table~\ref{tab:tg} by~(\ref{expr:tg}).
\begin{table}[htdp]
\caption{The value $\tg(n,k)$ (the minimum integer $\tg$
such that there exists an $N_{\tg}$-embeddable $(n,k)$-graph) for $1\le n\le 8$ and $0\le k\le 8$.}
\begin{center}
\begin{tabular}{|c|c|c|c|c|c|c|c|c|c|}
\hline
&0&1&2&3&4&5&6&7&8\\
\hline
1&1&1&2&5&10&15&22&31&40\\
\hline
2&1&1&4&7&12&19&26&35&46\\
\hline
3&1&2&5&10&15&22&31&40&51\\
\hline
4&1&4&7&12&19&26&35&46&57\\
\hline
5&2&5&10&15&22&31&40&51&64\\
\hline
6&4&7&12&19&26&35&46&57&70\\
\hline
7&5&10&15&22&31&40&51&64&77\\
\hline
8&7&12&19&26&35&46&57&70&85\\
\hline
\end{tabular}
\end{center}
\label{tab:tg}
\end{table}

Again, the first two rows determine the whole table
in the way that 
\[
\mu(n,N_{\tg})=\max\bigl(0,\,\mu(n-2,N_{\tg})-1\bigr).
\]
We obtain Table~\ref{tab:tmu} from Table~\ref{tab:tg}.

\begin{table}[htdp]
\caption
{The value $\mu(n,N_{\tg})$ 
(the minimum $k$ such that there is no $N_{\tg}$-embeddable $(n,k)$-graphs) 
for $1\le n\le 8$ and $1\le \tg\le 16$.}
\begin{center}
\begin{tabular}{|c|c|c|c|c|c|c|c|c|c|c|c|c|c|c|c|c|c|}
\hline
&1&2&3&4&5&6&7&8&9&10&11&12&13&14&15&16\\
\hline
1&2&3&3&3&4&4&4&4&4&5&5&5&5&5&6&6\\
\hline
2&2&2&2&3&3&3&4&4&4&4&4&5&5&5&5&5\\
\hline
3&1&2&2&2&3&3&3&3&3&4&4&4&4&4&5&5\\
\hline
4&1&1&1&2&2&2&3&3&3&3&3&4&4&4&4&4\\
\hline
5&0&1&1&1&2&2&2&2&2&3&3&3&3&3&4&4\\
\hline
6&0&0&0&1&1&1&2&2&2&2&2&3&3&3&3&3\\
\hline
7&0&0&0&0&1&1&1&1&1&2&2&2&2&2&3&3\\
\hline
8&0&0&0&0&0&0&1&1&1&1&1&2&2&2&2&2\\
\hline
\end{tabular}
\end{center}
\label{tab:tmu}
\end{table}
Along the same way, we obtain the general formula for the number~$\tilde\mu(n,N_{\tg})$.
This completes the proof.
\end{proof}

In terms of Euler characteristic, the above answer can be stated as Theorem~\ref{thm:nk} does.

\section*{Appendix: a proof of Theorem~\ref{thm:fc}}

One may rediscover Formula~(\ref{ans:fc}) from Theorem~\ref{thm:ans:nk2}.
The result $\rho(S_0)=5$ can be read directly from the column $k=0$ of Table~\ref{tab:g}.
Let $g\ge 1$ below.
From the definitions of $\rho(\Sigma)$ and $g(n,k)$, we see that $\rho(S_g)=n$ if and only if
\[
g(n-1,0)\le g\le g(n,0)-1,
\]
that is, if and only if
\[
\biggl\lceil{(n-2)(n-3)\over 12}\biggr\rceil
\le g\le
\biggl\lceil{(n-1)(n-2)\over 12}\biggr\rceil-1.
\]
By Lemma~\ref{lem:ineq:ceil}, the above inequalities are equivalent respectively to 
\[
{5-\sqrt{48g+1}\over 2}
\le n\le
{5+\sqrt{48g+1}\over 2}
\]
and
\[
n<{3-\sqrt{48g+1}\over 2}
\qquad\text{or}\qquad
n>{3+\sqrt{48g+1}\over 2}.
\]
It follows that 
\[
\rho(S_g)=n
=\biggl\lfloor{5+\sqrt{48g+1}\over 2}\biggl\rfloor
=\biggl\lfloor{5+\sqrt{49-24\chi}\over 2}\biggr\rfloor.
\]
Along the same line, we can show that 
\[
\rho(N_{\tg})
=\biggl\lfloor{5+\sqrt{24\tg+1}\over 2}\biggl\rfloor
=\biggl\lfloor{5+\sqrt{49-24\chi}\over 2}\biggr\rfloor.
\]
This proves Theorem~\ref{thm:fc}.

\end{document}